\documentclass[11pt]{article} 
\newcommand{\remove}[1]{}
\newcommand{\ignore}[1]{} 
\usepackage{amsmath,amssymb, amsthm,hyperref}

\newtheorem{thm}{Theorem}[section]
\newtheorem{claim}[thm]{Claim}
\newtheorem{lem}[thm]{Lemma}
\newtheorem{define}[thm]{Definition}
\newtheorem{cor}[thm]{Corollary}

\def\F{{\mathbb{F}}}

\def\R{{\mathbb{R}}}

\def\C{{\mathbb{C}}}

\def\_{\,\,\,\,\,}

\newcommand{\nospace}[1]{}

\def\supp{\textsf{supp}}

\def\rank{\textsf{rank}}

\def\poly{\textsf{poly}}

\def\adim{\textsf{adim}}

\newcommand{\eps}{\epsilon}

\begin{document}

\title{A quantitative variant of the multi-colored Motzkin-Rabin theorem}

\author{ Zeev Dvir \thanks{Department of Computer Science and Department of Mathematics, Princeton University, Princeton NJ.
Email: \texttt{zeev.dvir@gmail.com}. Research partially
supported by NSF grants CCF-0832797, CCF-1217416 and by the Sloan fellowship.} \and Christian Tessier-Lavigne \thanks{Department of Mathematics, Princeton University, Princeton NJ. Email: \texttt{ctessierlavigne@gmail.com}}}
\date{}
\maketitle

\begin{abstract} 
We prove a quantitative version of the multi-colored Motzkin-Rabin theorem in the spirit of \cite{BDWY12}: Let $V_1,\ldots,V_n \subset \R^d$ be  $n$ disjoint sets of points (of $n$ `colors'). Suppose that for every $V_i$ and every point $v \in V_i$ there are at least $\delta |V_i|$ other points $u \in V_i$ so that the line connecting $v$ and $u$ contains a third point of another color. Then the union of the points in all $n$ sets is contained in a subspace of dimension bounded by a function of $n$ and $\delta$ alone.
\end{abstract}

\section{Introduction}

The Motzkin-Rabin (MR) theorem (see \cite{BM90}) states that in a non-collinear set of points in the Euclidean plane, each colored blue or red, there always exists a monochromatic line (a line passing through at least two points and all points on the line are of the same color). Another way to state this theorem uses the following definition which we shall later generalize.

\begin{define}[MR configuration]\label{def-MRsimple}
Let $V_1,V_2 \subset \R^2$ be disjoint, finite sets of points in the plane. The pair $V_1,V_2$ is called an {\em MR-configuration} if every line $L$ with $|L \cap (V_1 \cup V_2)| \geq 2$ must intersect both sets $V_1$ and $V_2$.
\end{define}

The Motzkin-Rabin theorem can now be stated equivalently as:

\begin{thm}[Motzkin-Rabin Theorem]\label{thm-MRsimple}
Let $V_1,V_2 \subset \R^2$ be an MR-configuration. Then all points in $V_1 \cup V_2$ must belong to a single line.
\end{thm} 

It is easy to see that one can replace $\R^2$ with $\R^d$ and that the theorem will still hold in this case (take a generic projection to the plane). This theorem answers a question first raised by Graham \cite{Grunbaum}. The first published proof of Theorem~\ref{thm-MRsimple} appears in \cite{Chakerian} though it was proved earlier (but never published) by Motzkin and Rabin  \cite{Grunbaum}. 

We will denote by $\adim(S)$ the dimension of the affine span (the smallest affine subspace containing the points) of a point set $S \subset \R^d$ and for a family of sets $S_1,\ldots,S_r$ we will write $\adim(S_1,\ldots,S_r) = \adim(S_1 \cup \ldots \cup S_r)$. Then, the conclusion of the MR theorem, namely all points in $V_1,V_2$  being on a line, can be stated as $\adim(V_1,V_2) \leq 1$. Hence, we can view the MR theorem as converting partial information about collinearity in the sets $V_1,V_2$ (the line through every pair of points of the same color contains a third point of a different color) into a global bound on the dimension of the entire configuration. A closely related theorem is the Sylvester-Gallai theorem which is a `one color' version of the MR theorem: in every non-collinear set of points there is a line containing only two of the points.

Shannon \cite{Sha74} (see also \cite{Bor82}) proved an $n$-color variant of this theorem showing that if a family of $n$ sets $V_1,\ldots,V_n$ spans $\R^n$ then they must define at least one monochromatic line. In this work we extend this result to the setting where the information about collinearities is only given for {\em many} of the lines passing through two points of the same color. To be precise we will give the following definition:

\begin{define}[$(\delta,n)$-MR configuration]\label{def-deltanMR}
	Let $V_1, V_2, ..., V_n$ be disjoint sets of points in $\R^d$, and let $V = V_1 \cup V_2 \cup ... \cup V_n$. We say that $V_1, V_2, ..., V_n$ is a \emph{$(\delta, n)$-MR  configuration} if for each $V_i$ and for each $v \in V_i$, there are at least $\delta |V_i|$  points $u \in V_i\setminus \{v\}$ for which the line determined by $v$ and $u$  contains a third point in $V \setminus V_i$.	For convenience we will always assume that  $|V_1| \geq |V_2| \geq ... \geq |V_n|$. 
\end{define}

Our main theorem gives a dimension bound for $(\delta,n)$-MR configuration that depends only on $n$ and $\delta$. We do not believe our bound to be tight and conjecture that a bound of $\poly(n/\delta)$ holds in general.

\begin{thm}[Main theorem]\label{thm-main}
Let $V = V_1, V_2, ..., V_n \subset \R^d$ be a $(\delta,n)$-MR configuration. Then, for any $0 < \eps < \delta$ we have  $$\adim(V) \leq \frac{C}{\eps^2} \cdot \left(1 + \frac{1}{\delta-\eps}\right)^n,$$ with $C>0$ an absolute constant\footnote{One could set $\eps = \delta/2$ to get a simpler (but worse, in some cases) bound.}.
\end{thm}

Theorem~\ref{thm-main} is a multi-colored version of recent results of \cite{BDWY12, DSW12}, which give a similar `$\delta$-version' of the Sylvester-Gallai theorem (\cite{BDWY12} also establishes the $n=2$ case of Theorem~\ref{thm-main}). In fact, our proof uses one of the main results of \cite{BDWY12, DSW12} as its  principal tool. This result, given below as Theorem~\ref{thm-rankdesign}, gives a lower bound on the rank of matrices whose pattern of zeros and non-zeros satisfies a certain `design-like' condition. As the results of \cite{BDWY12,DSW12} work also over the complex numbers, our results (in particular, Theorem~\ref{thm-main}) hold also when one replaces $\R^d$ with $\C^d$ (with the same bounds).

In the next section we state some preliminaries from \cite{BDWY12,DSW12} that will be used in the proof of Theorem~\ref{thm-main}. The proof itself is given in Section~\ref{sec-proof}.

\section{Preliminaries}

The main tool in the proof is a rank lower bound for {\em design-matrices} defined in \cite{BDWY12}. For a vector $R \in \F^n$ we denote the {\em support} of $R$ by $\supp(R) = \{ i \in [n]\,\,|\,\, R_i \neq 0 \}$.

\begin{define}[Design matrix]\label{def-designmatrix}
Let $A$ be an $m \times n$ matrix over a field $\F$. Let $R_1,\ldots,R_m \in \F^n$ be  the rows of $A$ and let $C_1,\ldots,C_n \in \F^m$ be the columns of $A$. We say that $A$ is a {\em $(q,k,t)$-design matrix} if the following three conditions are satisfied:
\begin{enumerate}
\item For all $i \in [m]$, $|\supp(R_i)| \leq q$.
\item For all $j \in [n]$, $|\supp(C_j)| \geq k$.
\item For all $j_1 \neq j_2 \in [n]$, $|\supp(C_{j_1}) \cap \supp(C_{j_2}) | \leq t$.
\end{enumerate}
\end{define}

The following is a quantitative improvement of a bound originally proved in \cite{BDWY12}.

\begin{thm}[{\cite{DSW12}}]\label{thm-rankdesign}
Let $A$ by an $m \times n$ complex  matrix. If $A$ is a $(q,k,t)$ design matrix then $$\rank(A) \geq  n - \frac{ntq(q-1)}{k}.$$
\end{thm}

Another  lemma we will use is the following lemma whose proof is a simple consequence of the existence of diagonal Latin squares.

\begin{lem}[{\cite[Lemma 2.1]{BDWY12}}]\label{lem-triples}
	Let $r \geq 3$. Then there exists a set $T \subset [r]^3$ of $r^2 - r$ triples that satisfies the following properties.
\begin{enumerate}
	\item Each triple $(t_1, t_2, t_3) \in T$ consists of three distinct elements.
	\item For each $i \in [r]$ there are exactly $3(r - 1)$ triples in $T$ that contain $i$ as an element.
	\item For every pair $i, j \in [r]$ of distinct elements there are at most 6 triples in $T$ which contain both $i$ and $j$ as elements.
\end{enumerate}
\end{lem}

\section{Proof of the main theorem}\label{sec-proof}

Before giving the proof of Theorem~\ref{thm-main} we prove some useful lemmas. The first is the technical heart of the proof and its proof utilizes the rank bound for design matrices (Theorem~\ref{thm-rankdesign}). In the following we will denote by $\dim(S)$ the dimension of the subspace spanned by a set $S$. Notice that, since $\adim(S) \leq \dim(S)$, we can bound $\dim(V)$ instead of $\adim(V)$.

\begin{lem}\label{lem-partition}
Let $V = \bigcup_{i=1}^n V_i$ be a $(\delta, n)$-MR  configuration in $\R^d$. Let $x, y$ be indices with $0 \leq x < y \leq n$. Let $P_1 = \bigcup_{i=1}^x V_i$, let $P_2 = \bigcup_{i=x+1}^y V_i$, and let $P_3 = \bigcup_{i=y+1}^n V_i$ ($P_1$ and $P_3$ might be empty if $x=0$ or $y=n$). Suppose that for some constants $c_1, c_2 > 0$ the following two inequalities hold:
\begin{eqnarray}
 |V_y| \geq c_1|P_2|, \label{eq-sizes1}\\
(\delta - c_2)|V_y| \geq |P_3|.\label{eq-sizes2}	
\end{eqnarray}
Then $\dim(P_2) \leq \dim(P_1) + 12/(c_1 c_2)$.
\end{lem}
\begin{proof}
	We start by noting that, since $|V_1| \geq |V_2| \geq \ldots \geq |V_n|$, inequalities (\ref{eq-sizes1}) and (\ref{eq-sizes2}) in the lemma statement, $|V_y| \geq c_1|P_2|$ and $(\delta - c_2)|V_y| \geq |P_3|$,  also hold when $|V_y|$ is replaced with $|V_i|$, for $i < y$.

	We will call a line $L$ {\em extraordinary} with respect to the   configuration $V$ if (1) $L$ passes through at least one point of $P_2$ and (2) $L$ passes through at least three points of $P_1 \cup P_2$. We will refer to the points of $P_1 \cup P_2$ that lie on some extraordinary line $L$ as the \emph{points associated with $L$} (such a line $L$ might contain additional points from $P_3$ which are not associated with it).

Let $L_1, L_2, ..., L_k$ be an enumeration of the extraordinary lines of our configuration and let $\ell_i$ denote the number of points associated with $L_i$, for $1 \leq i \leq k$.

	For each extraordinary line $L_i$ we  construct, using Lemma~\ref{lem-triples}, a set  $T_i$ of $\ell_i^2 - \ell_i$ triples of points so that (1) each triple in $T_i$ consists of three distinct points associated with $L_i$; (2) for any point $v$ associated with $L_i$, there are exactly $3(\ell_i - 1)$ triples in $T_i$ that contain $v$; and (3) for any two points $u \neq v$ associated with $L_i$, there are at most 6 triples in $T_i$ that contain both $u$ and $v$. Let $$T = \bigcup_{i=1}^k T_i.$$

	Next, let $m = |V|$ and let $M$ be the $m \times d$ matrix whose rows are defined by the points of $V$ (in some choice of coordinates for $\R^d$). We will now define a matrix $A$ that will satisfy  $A\cdot M = 0$. Each triple in $T$ will correspond to one row of $A$. Every triple $t = (t_1, t_2, t_3) \in T$ consists of three distinct points in $P_1 \cup P_2$ that are collinear. Since they are collinear, there are coefficients $h_1, h_2, h_3$, not all zero, such that $$h_1t_1 + h_2t_2 + h_3t_3 = 0$$ (treating the points as vectors). We set the  $t$'th row of $A$ to have entries $h_1,h_2,h_3$ in the positions corresponding to the three points $t_1,t_2,t_3$ (we can do that since the columns of $A$ are indexed by $V$) and zero elsewhere. Observe that $A$ is a $|T| \times m$ matrix, since there is a bijection between the elements of $T$ and the rows of $A$. Since the product of any row of $A$ with $M$ is $0$, we must also have that $$A\cdot M = 0.$$

	There is a bijection between the rows of the matrix $M$ and the points in the set $V$. Therefore, any subset of the set $V$ corresponds to a submatrix of the matrix $M$, obtained by taking only those rows that correspond to the points in the subset. Let $M_1$ denote the submatrix of $M$ corresponding to the point set $P_1$, and likewise let $M_2$ and $M_3$ be the submatrices corresponding to $P_2$ and $P_3$. Let $A_1$ be the submatrix of $A$ obtained by taking those \emph{columns} of $A$ whose indices match the indices of the \emph{rows} of $M_1$ (that is, with indices corresponding to elements of $P_1$). Define $A_2$ and $A_3$ analogously (with columns in $P_2$ and $P_3$ respectively). Observe that $A_1M_1, A_2M_2,$ and $A_3M_3$ are all valid matrix products, and that $$A_1M_1 + A_2M_2 + A_3M_3 = AM =  0.$$

	From the definition of the matrix $A$ we have that the column corresponding to any given point in $P_3$ contains only $0$'s; therefore $A_3 = 0$, and so $A_3M_3 = 0$. Hence  $A_1M_1 + A_2M_2 = 0$ which gives
\begin{equation}\label{eq-dimA2M2}
\rank(A_2M_2) = \rank(A_1M_1) \leq  \rank(M_1) = \dim(P_1).
\end{equation}
(If $|P_1|$ is empty we get $A_2M_2 = 0$ and the rest of the proof is the same).
	
We now claim that:
 
\begin{claim}
$A_2$ is a $(3, 3c_1c_2|P_2|, 6)$-design matrix.
\end{claim}
\begin{proof}
	By the construction of $A$ each row contains at most three non-zero terms. Since $A_2$ is a submatrix of $A$, each row of $A_2$ can contain at most three non-zero terms. Similarly, by the construction of $A$, any two columns can share at most six non-zero locations; and again this holds for $A_2$ as well. Finally, we claim that each column of $A_2$ contains at least $3c_1c_2|P_2|$ non-zero entries. 

		Consider a column $C$ of $A_2$. This column corresponds to a point $p$ in $P_2$. The number of non-zero entries of $C$ is exactly equal to the number of triples in $T$ that contain the point $p$. Suppose that $p \in V_i \subset P_2$ for some $i$. We claim that there must be at least $\delta|V_i| - |P_3|$  points $q \neq p$ that lie on extraordinary lines through $p$. Observe that this quantity is at least $c_2|V_i|$ by  inequality (\ref{eq-sizes2}). Indeed, there are at least $\delta |V_i|$ points $q \neq p$ in $V_i$, for which the line through $q, p$ contains a point from some $V_j$, with $j \not= i$, because the configuration is $(\delta, n)$-MR. Let us denote this set of at least $\delta |V_i|$ points by $S$. For each point $q$ in $S$, either the line through $q, p$ contains a third point from $P_1 \cup P_2$, and is therefore an extraordinary line, \emph{or} (1) it contains no other points from $P_1 \cup P_2$, and (2) it contains some point $r$ from $P_3$.

		Thus, each point $q \in S$ that is not associated with any of the extraordinary lines passing through $p$ corresponds to some point $r \in P_3$. Since no two $q_1 \neq  q_2 \in S$ can correspond to the same $r$, at most $|P_3|$ of the points in $S$ are \emph{not} associated with any of the extraordinary lines passing through $p$. Thus, the remaining $\delta|V_i| - |P_3|$ points are associated with one of the extraordinary lines passing through $p$.
		
Now, if a given extraordinary line $L$ passes through $p$, and if there are $\ell$ points associated with $L$ besides $p$, then that line contributes $3\ell$ triples to $T$ that contain $p$. Therefore, since we showed that there are at least $c_1 c_2 |P_2|$ points other than $p$ that lie on the extraordinary lines passing through $p$, there must be at least $3c_1 c_2|P_2|$ triples in $T$ that contain $p$.

We conclude that that the point $p \in V_i$ is in  at least $3c_2|V_i|$ triples; and since $|V_i| \geq c_1|P_2|$ (by inequality (\ref{eq-sizes1})), this quantity is at least $3c_1c_2|P_2|$ such points and so $A_2$ is indeed a $(3, 3c_1c_2|P_2|, 6)$ design matrix as claimed.
\end{proof}

Applying Theorem~\ref{thm-rankdesign} we have that $$\rank(A_2) \geq |P_2| - 12/(c_1c_2).$$ Now, using basic linear algebra, we get that $$\rank(A_2M_2) \geq \rank(M_2) - (|P_2| - \rank(A_2)) \geq \rank(M_2)  - 12/(c_1c_2).$$  Using Eq.~(\ref{eq-dimA2M2}) we immediately get $$\rank(M_2) \leq \dim(P_1) + 12/(c_1c_2),$$ which implies $\dim(P_2) \leq \dim(P_1) + 12/(c_1c_2)$ as was required. This completes the proof of Lemma~\ref{lem-partition}.
\end{proof}

To state the next lemma we will need the following definition.

\begin{define}[$\eps$-large and $\eps$-small indices]\label{def-epslarge}
Let $V_1,\ldots,V_n \subset \R^d$ be a $(\delta,n)$-MR configuration and let $c_\eps = 1/(\delta-\eps)$ with $0 < \eps < \delta$ some real number. We call an index $k \in [n]$ an {\em $\eps$-large} index if  $$|V_k| \geq c_\epsilon(|V_{k + 1}| + |V_{k + 2}| + ... + |V_n|),$$ otherwise we say that $k$ is {\em $\eps$-small}. By convention, we say that $n$ is always $\eps$-large.
\end{define}

\begin{lem}\label{lem-epslarge}
Let $V_1, V_2, ..., V_n \subset \R^d$ be a $(\delta, n)$-MR configuration, and suppose $x$ and $y$ are integers with $0 \leq x < y \leq n$ such that $y$ is an $\epsilon$-large index, and each of the indices $x + 1, x + 2, ..., y - 2, y - 1$ is $\epsilon$-small. Then, for each $i$ with $0 \leq i \leq y - x - 1$ we have  $$\sum_{j \geq y - i} |V_j| \leq 2(1 + c_\epsilon)^i \cdot |V_y|.$$
\end{lem}
\begin{proof}
We will prove the lemma by induction on $i$. To prove the base case, $i = 0$, we need to show that $$\sum_{j \geq y} |V_j| \leq 2 |V_y|.$$ Since $y$ is an $\epsilon$-large index, we have  $$|V_y| \geq c_\epsilon \sum_{j > y} |V_j|$$ and  so $\sum_{j > y} |V_j| \leq 1/c_\epsilon |V_y|.$ By adding $|V_y|$ to both sides we immediately have that $$\sum_{j \geq y} |V_j| \leq (1 + 1/c_\epsilon) |V_y|$$ which gives the desired bound since $c_\eps >1$ and so $1 + 1/c_\eps < 2$.

	Now suppose the claim holds for $i = k$. We wish to show that it also holds for $i = k + 1$, assuming that $k + 1 \leq y - x - 1$. From the induction we have that $$\sum_{j \geq y - k} |V_j| \leq 2(1 + c_\epsilon)^k |V_y|.$$ We also know that $y- (k + 1)$ is an $\epsilon$-small index, so $$|V_{y - (k + 1)}| \leq c_\epsilon \sum_{j \geq y - k} |V_j|.$$  Substituting the first inequality into the second gives $$|V_{y - (k + 1)}| < 2 c_\epsilon (1 + c_\epsilon)^k |V_y|.$$ Then adding this inequality to the first inequality yields the desired result.
	
\end{proof}

\begin{cor}\label{cor-epslarge}
Under the same notations and conditions as Lemma~\ref{lem-epslarge}, we have:
$$|V_y|\geq \frac{1}{2(1 + c_\epsilon)^{y - x - 1}}\sum_{j=x+1}^{y}|V_j|.$$
\end{cor}
\begin{proof}
	Apply Lemma~\ref{lem-epslarge} with $i = y - x - 1$ to get that $$\sum_{j=x+1}^{n}|V_j| \leq 2(1 + c_\epsilon)^{y - x - 1} \cdot |V_y|,$$ hence $$\sum_{j=x+1}^{y}|V_j| \leq 2(1 + c_\epsilon)^{y - x - 1} \cdot |V_y|$$ and the corollary follows.
\end{proof}  

\subsection{ Proof of Theorem~\ref{thm-main}}

Let $d_1 < d_2 < ... < d_k=n$ be the $\epsilon$-large indices determined by $V$  (see Definition~\ref{def-epslarge}) and let us define $d_0 = 0$. We  define $$W_1 = V_1 \cup V_2 \cup ... \cup V_{d_1};$$ $$W_2 = V_{{d_1} + 1} \cup V_{{d_1} + 2} \cup ... \cup V_{d_2}$$ etc. for $1 \leq i \leq k$. Let  $$m_i = \dim(W_1 \cup W_2 \cup ... \cup W_i),$$ for $1 \leq i \leq k$ and set $m_0 = 0$.

Consider $W_i$ for some  $1 \leq i \leq k$. Since $d_i$ is an $\epsilon$-large index, we have that $$|V_{d_i}| \geq c_\epsilon(|V_{{d_i} + 1}| + |V_{{d_i} + 2}| + ... + |V_n|)$$ with $c_\epsilon = 1/(\delta - \epsilon)$. Hence $$(\delta - \epsilon)|V_{d_i}| \geq |V_{{d_i} + 1}| + |V_{{d_i} + 2}| + ... + |V_n|.$$ Furthermore, each of $d_{i - 1} + 1, d_{i - 1} + 2, ..., d_i - 2, d_i - 1$ are $\epsilon$-small indices, so by Corollary~\ref{cor-epslarge} we have  $$|V_{d_i}| \geq \frac{1}{2(1 + c_\epsilon)^{d_i - d_{i - 1} - 1} }(|V_{d_{i - 1} + 1}| + |V_{d_{i - 1} + 2}| + ... + |V_{{d_i} - 1}| + |V_{d_i}|).$$

Therefore our configuration satisfies the conditions of Lemma~\ref{lem-partition}, with $x = d_{i - 1}$, $y = d_i$, $c_1 = \frac{1}{2(1 + c_\epsilon)^{d_i - d_{i - 1} - 1}}$, and $c_2 = \epsilon$. For these values of $x$ and $y$, the set $P_1$ defined in the lemma equals $W_1 \cup W_2 \cup ... W_{i - 1}$, and the set $P_2$ equals $W_i$. Therefore we get that  $$\dim(W_i) \leq m_{i - 1} + (24/\eps)\cdot (1 + c_\epsilon)^{d_i - d_{i - 1} - 1}.$$ Now, since $m_i \leq m_{i - 1} + \dim(W_i)$, we have that $$m_i \leq 2m_{i - 1} + (24/\eps) \cdot (1 + c_\epsilon)^{d_i - d_{i - 1} - 1}.$$

\begin{claim}
For all $0 \leq i \leq k$ we have $$m_i \leq \frac{24}{\eps} \sum_{1 \leq j \leq i} 2^{i - j}(1 + c_\epsilon)^{d_j - d_{j - 1} - 1}.$$
\end{claim}
\begin{proof}
We prove the claim by induction on $i$. The base case, $i = 0$, holds since $m_0 = 0$. Suppose the claim holds for $i = h$ and consider the case $i = h + 1$. By induction we have that $$m_h \leq \frac{24}{\eps} \sum_{1 \leq j \leq h} 2^{h - j}(1 + c_\epsilon)^{d_j - d_{j - 1} - 1}.$$ We also showed that $$m_{h + 1} \leq 2m_h + \frac{24}{\eps}(1 + c_\epsilon)^{d_{h + 1} - d_h - 1}. $$ Substituting the first inequality into the second we find that $$m_{h + 1} \leq \frac{24}{\eps} \sum_{1 \leq j \leq h} 2^{h + 1 - j}(1 + c_\epsilon)^{d_j - d_{j - 1} - 1} + \frac{24}{\eps}(1 + c_\epsilon)^{d_{h + 1} - d_h - 1} $$ which gives the desired result. 
\end{proof}

Using the claim for $i=k$ we get $$m_k \leq \frac{24}{\eps} \sum_{1 \leq j \leq k} 2^{k - j}(1 + c_\epsilon)^{d_j - d_{j - 1} - 1}.$$ Observe that for all $j$, $d_j - d_{j - 1} \leq n - k + 1$. This follows from the fact that the $d_j$ are strictly increasing, $d_0 = 0$, and $d_k = n$. Therefore, the summand $2^{k - j}(1 + c_\epsilon)^{d_j - d_{j - 1} - 1}$ is at most $2^{k - j}(1 + c_\epsilon)^{n - k}$, which in turn is at most $(1 + c_\epsilon)^n \cdot \left(\frac{2}{1 + c_\epsilon}\right)^k$. Adding these together we get that $$m_k \leq \frac{24\cdot k}{\eps}  (1 + c_\epsilon)^n \cdot \left(\frac{2}{1 + c_\epsilon}\right)^k.$$ Observe that, since $c_\eps = 1/(\delta - \eps) > 1/(1-\eps) > 1+ \eps$, we have  $2/(1+c_\eps) < 2/(2+\eps)$ and so we get that
$$m_k \leq \frac{24\cdot k}{\eps} \left(\frac{2}{2+\eps}\right)^k (1 + c_\epsilon)^n.$$

The expression $ \frac{24\cdot k}{\eps} \left(\frac{2}{2+\eps}\right)^k$ is maximized when $k = -\frac{1}{\ln(2/(2+\eps))} = O(1/\eps)$ and so we get
$$m_k \leq  \frac{C}{\eps^2}\cdot (1 + c_\eps)^n$$
For some absolute constant $C$. Since $m_k = \dim(W_1 \cup W_2 \cup ... \cup W_n) = \dim(V_1,\ldots,V_n)$, the proof of the theorem is complete.

\bibliographystyle{alpha}
\bibliography{MR}  % sigproc.bib is the name of the Bibliography in this case
% ACM needs 'a single self-contained file'!
%

\end{document}